\documentclass[11pt]{amsart}%
\usepackage{amsfonts,tikz}
\usepackage{amsmath}%
\usepackage{amssymb}%
\usepackage{graphicx}
\providecommand{\U}[1]{\protect\rule{.1in}{.1in}}
\newcommand{\tph}{2\pi\hslash}
\newcommand{\rn}{\mathbb{R}^n}
\newcommand{\rnn}{\mathbb{R}^{2n}}
\newcommand{\sn}{\mathcal{S}(\mathbb{R}^n)}
\newcommand{\spn}{\mathcal{S}'(\mathbb{R}^n)}
\newcommand{\snn}{\mathcal{S}(\mathbb{R}^{2n})}
\newcommand{\spnn}{\mathcal{S}'(\mathbb{R}^{2n})}

\newtheorem{theorem}{Theorem}

\newtheorem{definition}[theorem]{Definition}
\newtheorem{example}[theorem]{Example}

\newtheorem{proposition}[theorem]{Proposition}
\newtheorem{remark}[theorem]{Remark}

\begin{document}

\title{On the Invertibility of Born-Jordan Quantization}
\author{Elena Cordero,  Maurice de Gosson and Fabio Nicola}
\address{Dipartimento di Matematica,
Universit\`a di Torino, via Carlo Alberto 10, 10123 Torino, Italy}
\address{University of Vienna, Faculty of Mathematics, Oskar-Morgenstern-Platz 1
A-1090 Wien, Austria}
\address{Dipartimento di Scienze Matematiche,
Politecnico di Torino, corso Duca degli Abruzzi 24, 10129 Torino,
Italy}

\email{elena.cordero@unito.it}
\email{maurice.de.gosson@univie.ac.at}
\email{fabio.nicola@polito.it}


\maketitle

\begin{abstract}
As a consequence of the Schwartz kernel Theorem, any linear continuous operator $\widehat{A}:$ $\mathcal{S}%
(\mathbb{R}^{n})\longrightarrow\mathcal{S}^{\prime}(\mathbb{R}^{n})$ can be written in Weyl form in a unique way, namely it is the Weyl quantization of a unique symbol $a\in  \mathcal{S}^{\prime}(\mathbb{R}^{2n})$. Hence, dequantization can always be performed, and in a unique way. Despite the importance of this topic in Quantum Mechanics and Time-frequency Analysis, the same issue for the Born-Jordan quantization seems simply unexplored, except for the case of polynomial symbols, which we also review in detail. In this paper we show that any operator $\widehat{A}$ as above can be written in Born-Jordan form, although the representation is never unique if one allows general temperate distributions as symbols. Then we consider the same problem when the space of temperate distributions is replaced by the space of smooth slowly increasing functions which extend to entire function in $\mathbb{C}^{2n}$, with a growth at most exponential in the imaginary directions. We prove again the validity of such a representation, and we determine a sharp threshold for the exponential growth under which the representation is unique. We employ techniques from the theory of division of distributions.
\end{abstract}

\section{Introduction}
Roughly speaking, quantization is the process of associating to a function or
distribution defined on phase space an operator. Historically, this notion
appears explicitly for the first time in Born and Jordan's foundational paper
\cite{bj} where they set out to give a firm mathematical basis to Heisenberg's
matrix mechanics. Born and Jordan's quantization scheme was strictly speaking
limited to polynomials in the variables $x$ and $p$; it was soon superseded by
another rule due to Weyl, and whose extension is nowadays the preferred
quantization in physics. However, it turns out that there is a recent regain in
interest in an extension of Born and Jordan's initial rule, both in Quantum Physics and Time-frequency Analysis. In fact, on the one hand it is the correct rule if one wants matrix and wave mechanics to be equivalent quantum theories (see the discussion in \cite{gofound}).
On the other hand, as a time-frenquency representation the Born-Jordan distribution has been proved to be surprisingly successful, because it allows to damp very well the unwanted ``ghost frequencies'', as shown in \cite{bogetal,turunen}. 
\par
The difference between Born--Jordan and Weyl quantization is most easily
apprehended on the level of monomial quantization: in dimension $n=1$ for any integers $r,s\geq0$
we have
\begin{equation}
\operatorname*{Op}\nolimits_{\mathrm{W}}(x^{r}p^{s})=\frac{1}{2^{s}}\sum
_{\ell=0}^{s}\binom{s}{\ell}\widehat{p}^{s-\ell}\widehat{x}^{r}\widehat
{p}^{\ell}=\frac{1}{2^{r}}\sum_{\ell=0}^{r}\binom{r}{\ell}\widehat{x}^{\ell
}\widehat{p}^{s}\widehat{x}^{r-\ell}\label{weylmono1}%
\end{equation}
(see \cite{mccoy}) and
\begin{equation}
\operatorname*{Op}\nolimits_{\mathrm{BJ}}(x^{r}p^{s})=\frac{1}{s+1}\sum
_{\ell=0}^{s}\widehat{p}^{s-\ell}\widehat{x}^{r}\widehat{p}^{\ell}=\frac
{1}{r+1}\sum_{\ell=0}^{r}\widehat{x}^{\ell}\widehat{p}^{s}\widehat{x}^{r-\ell
}\label{bjmono1}%
\end{equation}
(see \cite{bj}). As usual here $\widehat{p}=-i\hslash\partial/\partial x$ and $\widehat{x}$ is the multiplication operator by $x$. The Born--Jordan scheme thus appears as being an
equally-weighted quantization, as opposed to the Weyl scheme:
$\operatorname*{Op}\nolimits_{\mathrm{BJ}}(x^{r}p^{s})$ is the average of all
possible permutations of the product $\widehat{x}^{r}\widehat{p}^{s}$.\par
 One can
extend the Weyl and Born--Jordan quantizations to arbitrary symbols
$a\in\mathcal{S}^{\prime}(\mathbb{R}^{2n})$ by defining the operators $\widehat{A}_{\mathrm{W}}=\operatorname*{Op}\nolimits_{\mathrm{W}}(a)$ and
$\widehat{A}_{\mathrm{BJ}}=\operatorname*{Op}\nolimits_{\mathrm{BJ}}(a)$: $\mathcal{S}(\mathbb{R}^{n})\to \mathcal{S}'(\mathbb{R}^{n})$ as
\begin{align}\label{intro3}
\widehat{A}_{\mathrm{W}}\psi & =\left(  \tfrac{1}{2\pi\hbar}\right)  ^{n}%
\int_{\mathbb{R}^{2n}}a_{\sigma}(z)\widehat{T}(z)\psi dz\\
\widehat{A}_{\mathrm{BJ}}\psi & =\left(  \tfrac{1}{2\pi\hbar}\right)  ^{n}%
\int_{\mathbb{R}^{2n}}a_{\sigma}(z)\Theta(z)\widehat{T}(z)\psi dz\nonumber
\end{align}
where $\psi
\in\mathcal{S}(\mathbb{R}^{n})$ and the integrals are to be understood in the distributional sense; here
$\widehat{T}(z_{0})=e^{-i(x_{0}\widehat{p}-p_{0}\widehat{x})/\hbar}$, $z_0=(x_0,p_0)$, is the
Heisenberg operator, $a_\sigma(z)=a_{\sigma}(x,p)=Fa(p,-x)$, with $z=(x,p)$, is the symplectic Fourier
transform of $a$, and $\Theta$ is Cohen's \cite{Cohen1} kernel function,
defined by
\[
\Theta(z)=\frac{\sin(px/2\hbar)}{px/2\hbar}%
\]
 ($\Theta(z)=1$ for $z=0$); see 
\cite{bogetal, TRANSAM, golu1} and
the references therein. The presence of the function $\Theta$ produces a smoothing effect (in comparison with the Weyl quantization) which is responsible of the superiority of the Born-Jordan quantization in several respects (\cite{bogetal,turunen}). However, although this effect is numerically evident it remains a challenging open problem to quantify it analytically.\par
 Now, it readily follows from the Schwartz kernel
Theorem that for every linear continuous operator $\widehat{A}:$ $\mathcal{S}%
(\mathbb{R}^{n})\longrightarrow\mathcal{S}^{\prime}(\mathbb{R}^{n})$ there
exists a unique $b\in\mathcal{S}^{\prime}(\mathbb{R}^{2n})$ such that
$\widehat{A}=\operatorname*{Op}\nolimits_{\mathrm{W}}(b)$. In other terms, dequantization can always be performed, and in a unique way. Instead the situation is
 more complicated for Born--Jordan operators. In fact, to prove that there
exists $a\in\mathcal{S}^{\prime}(\mathbb{R}^{2n})$ such that $\widehat
{A}=\operatorname*{Op}\nolimits_{\mathrm{BJ}}(a)$ one has to solve a division
problem, namely to find a distribution $a$ such that \[
b_{\sigma}= \Theta a_{\sigma
}.
\] The existence of such a symbol $a$ is far from being obvious because of the
zeroes of $\Theta$. It moreover turns out, as we shall see, that the solution is not even unique. The aim of this paper is to investigate these issues.\par
The problem of the division of temperate distributions by smooth functions is in general a very subtle one, even in the presence of simple zeros \cite{atiyah,bonet,hormander0,lang,loja,schwartz}. The basic idea here is of course that the space $\mathcal{S}^{\prime}(\mathbb{R}^{n})$ contains (generalized) functions rough enough to absorb the singularities and the loss of decay arising in the division by $\Theta$, and we have in fact the following result (Theorem \ref{mainteo1}). \par\medskip
{\it Every linear continuous operator $\widehat{A}:$ $\mathcal{S}%
(\mathbb{R}^{n})\longrightarrow\mathcal{S}^{\prime}(\mathbb{R}^{n})$ can be written in Born-Jordan form, i.e. there
exists $a\in\mathcal{S}^{\prime}(\mathbb{R}^{2n})$ such that
$\widehat{A}=\operatorname*{Op}\nolimits_{\mathrm{BJ}}(a)$.} \par\medskip
We provide two proofs of this result. One is completely elementary and constructive. The other is shorter and is based on the machinary of a priori estimates developed in \cite{hormander0} to prove that the division by a (non identically zero) polynomial is always possible in $\mathcal{S}'(\mathbb{R}^n)$. Actually the reader familiar with \cite{hormander0} will notice that the tools used there are excessively sophisticated for our purposes: after all the function $\Theta(z)$ is the composition of the harmless ``{\rm sinc}" function with the polynomial $px$, and indeed we will show that a suitable change of variables will reduce matters to the problem of division by the ``sinc'' function. 
One can also rephrase this result as follows. \par\medskip
{\it  The map $\mathcal{S}^{\prime}(\mathbb{R}^{2n})\to \mathcal{S}^{\prime}(\mathbb{R}^{2n})$
\begin{equation}\label{map00}
a\longmapsto a\ast\Theta_\sigma,
\end{equation}
which gives the Weyl symbol of an operator with Born-Jordan symbol $a$, is surjective.}\par\medskip
It is important to observe that the above representation of the operator $\widehat{A}$ is never unique: if $\widehat{A}=\operatorname*{Op}\nolimits_{\mathrm{BJ}}(a)$ and $z_0\in\mathbb{R}^{2n}$ verifies $\Theta(z_0)=0$ then the symbol $a(z)+e^{\frac{i}{\hslash}\sigma(z_0,z)}$ gives rise to the {\it same} operator $\widehat{A}$; see Example \ref{rem2} below. \par
As one may suspect, imaginary-exponential symbols play an important role in the discussion. In fact, the function $e^{\frac{i}{\hslash}\sigma(z_0,z)}$ turns out to be the Weyl symbol of the operator $\widehat{T}(z_0)$ and, in general, any operator can be regarded as a superposition of $\widehat{T}(z)$'s; cf.\ \eqref{intro3}. This suggests the study of the map \eqref{map00} in spaces of smooth temperate functions which extend to entire functions in $\mathbb{C}^{2n}$, with a  growth at most exponential in the imaginary directions. To be precise, for $r\geq0$, let $\mathcal{A}_r$ be the space of smooth functions $a$ in $\mathbb{R}^{2n}$ that extend to entire functions $a(\zeta)$ in $\mathbb{C}^{2n}$ and satisfying the estimate
\[
|a(\zeta)|\leq C(1+|\zeta|)^N\exp\Big(\frac{r}{\hslash}|{\rm Im}\,\zeta|\Big),\quad \zeta\in \mathbb{C}^{2n},
\]
for some $C,N>0$. For example the symbol $e^{\frac{i}{\hslash}\sigma(z_0,z)}$ belongs to $\mathcal{A}_r$ with $r=|z_0|$, whereas $\mathcal{A}_0$ is nothing but the space of polynomials in phase space. Then we have the following result (Proposition \ref{pro6} and Theorem \ref{mainteo2}). \par\medskip
{\it The map $\mathcal{A}_r\to \mathcal{A}_r$ in \eqref{map00} is surjective for every $r\geq0$. It is also one to one (and therefore a bijection) if and only if $0\leq r<\sqrt{4\pi\hslash}$.}\par\medskip
A detailed study for polynomial symbols (case $r=0$) will be carried out in Section \ref{polynomial}, where an explicit formula for the inverse map is provided. The above threshold $r=\sqrt{4\pi\hslash}$ seems to have an interesting physical interpretation in terms of the Heisenberg uncertainty principle/symplectic capacity, and will be explored elsewhere. \par\medskip
The present paper represents a first step of a project in the undestanding of the Born-Jordan quantization within the general framework of the temperate distributions. In  fact, in view of the role of the Born-Jordan quantization in Time-frequency Analysis it would be certainly interesting to study the invertibility issue in smaller spaces of functions and distributions arising in Fourier Analysis, such as (weighted) Sobolev spaces, modulation spaces, Wiener amalgam spaces, cf.\ \cite{Birkbis,grochenig1}. Also, due to the above mentioned smoothing effect one expects Born-Jordan operators enjoy better continuity properties than those known for Weyl operators (cf.\ \cite{boulk,cordero1,cordero2,cordero3,cunanan,grochenig1,grochenig2,kobayashi,stein,sugimoto1,sugimoto2,sugimoto3,toft1,toft2} and the references therein). Some results in this direction were already obtained in \cite{bogetal,turunen}, and we plan to continue this study in a subsequent work \cite{cgnp}. \par\medskip
The paper is organized as follows. In Section 2 we collected some preliminary results on the division of distributions. Weyl and Born-Jordan quantizations are then introduced in Section 3, whereas Section 4 is devoted to a detailed analysis of the case of polynomial symbols. In Section 5 we address the problem of the invertibility of the map \eqref{map00} in the space $\mathcal{S}'(\mathbb{R}^{2n})$ and in the spaces $\mathcal{A}_r$ defined above.

\section{Notation and preliminary results}
\subsection{Notation}
We will use multi-index notation: $\alpha=(\alpha_{1},...,\alpha_{n}%
)\in\mathbb{N}^{n}$, $|\alpha|=\alpha_{1}+\cdot\cdot\cdot+\alpha_{n}$, $x^\alpha=x_1^{\alpha_1}\cdots x_n^{\alpha_n}$, 
$\partial_{x}^{\alpha}=\partial_{x_{1}}^{\alpha_{1}}\cdots
\partial_{x_{n}}^{\alpha_{n}}$. \par
As usual $\sn$ denotes the Schwartz space of smooth functions $\psi$ in $\rn$ such that
\begin{equation}\label{seminorm}
\|\psi\|_N:=\sup_{|\alpha|+|\beta|\leq N}\sup_{x\in\rn}|x^\alpha\partial_x^\beta \psi(x)|<\infty 
\end{equation}
for every $N\geq0$. This is a Fr\'echet space, endowed with the above seminorms. We denote by $\spn$ the dual space of temperate distributions. 

 We denote by $\sigma$ the standard symplectic
form on the phase space $\mathbb{R}^{2n}\equiv\mathbb{R}^{n}\times
\mathbb{R}^{n}$; the phase space variable is denoted by $z=(x,p)$. By definition
$\sigma(z,z^{\prime})=Jz\cdot z^{\prime}=p\cdot x'-x\cdot p'$ (with $z'=(x',p')$), where 
\[J=%
\begin{pmatrix}
0_{n\times n} & I_{n\times n}\\
-I_{n\times n} & 0_{n\times n}%
\end{pmatrix}.
\]

We will use the notation $\widehat{x}_{j}$ for the operator of multiplication
by $x_{j}$ and $\widehat{p}_{j}=-i\hslash\partial/\partial x_{j}$. These
operators satisfy Born's canonical commutation relations $[\widehat{x}%
_{j},\widehat{p}_{j}]=i\hslash$.

The Fourier transform of a function $\psi(x)$ in $\rn$ is defined as 
\[
F\psi(p)= \left(\tfrac{1}{\tph}\right)^{n/2}\int_{\rn} e^{-\tfrac{i}{\hslash}p x} \psi(x)\, dx,
\]
where $px=p\cdot x=\sum_{j=1}^n p_j x_j$, and the symplectic Fourier transform of a function $a(z)$ in phase space $\rnn$ is
\[
F_\sigma a(z)=a_\sigma(z)= \left(\tfrac{1}{\tph}\right)^n\int_{\rnn} e^{-\tfrac{i}{\hslash}\sigma(z,z')} a(z')\, dz'.
\]
We observe that the symplectic Fourier transform is an involution, i.e. $(a_\sigma)_\sigma=a$, and moreover $a_\sigma(z)= Fa(J z)$. We will also use frequently the important relation
\begin{equation}\label{eq1}
(a\ast b)_\sigma=(\tph)^na_\sigma b_\sigma.
\end{equation}
\subsection{Compactly supported distributions}
We recall the Paley-Wiener-Schwartz Theorem (see e.g.\ \cite[Theorem 7.3.1]{hormander}). 
\begin{theorem}
For $r\geq0$, let $B_r$ be the closed ball $|x|\leq r$ in $\rn$. If $u$ is a distribution with compact support in $B_{r}$ then its (symplectic) Fourier transform extends to an entire analytic function in $\mathbb{C}^n$ and satisfies
\[
|Fu(\zeta)|\leq C(1+|\zeta|)^N\exp\Big(\frac{r}{\hslash}|{\rm Im}\,\zeta|\Big)
\]
for some $C,N>0$.\par
Conversely, every entire analytic function satisfying an estimate of this type is the (symplectic) Fourier transform of a distribution supported in $B_{r}$. 
\end{theorem}
\subsection{Division of distributions}\label{division} We  begin with a technical result, inspired by \cite[Theorem VII, page 123]{schwartz}, which will be used in the sequel.\par
Recall the definition of the seminorm $\|\varphi\|_N$  in \eqref{seminorm}. 
\begin{proposition}\label{pro2} Let $v\in\spn$ and $\chi\in C^\infty_c(\mathbb{R})$. For every $t\in \mathbb{R}$ there exists a distribution $u_t\in\spn$ satisfying 
\begin{equation}\label{eq0}
(x_n-t)u_t=\chi(x_n-t)v.
\end{equation}
Moreover $u_t$ can be chosen so that
\begin{equation}\label{eq1}
|u_t(\varphi)|\leq C\|\varphi\|_N\quad \forall\varphi\in\sn
\end{equation}
for some constants $C,N>0$ independent of $t$.
\end{proposition}
\begin{proof}
We define $u_t$ as follows. Write $x=(x',x_n)$, and let $\varphi\in\sn$. We have 
\[
\varphi(x)=\varphi(x',t)+(x_n-t)\tilde{\varphi}_t(x)
\]
with 
\[
\tilde{\varphi}_t(x)=\int_0^1 \partial_{x_n}\varphi(x',t+\tau(x_n-t))\, d\tau.
\]
Observe that $\chi(x_n-t)\tilde{\varphi}_t \in\sn$ (whereas $\tilde{\varphi}_t$ is not Schwartz in the $x_n$ direction, in general) and define
\[
u_t(\varphi)=v(\chi(x_n-t)\tilde{\varphi}_t).
\]
It is easy to see that $u_t$ is in fact a temperate distribution: since $v\in \spn$ we have 
\[
|u_t(\varphi)|\leq C\|\chi(x_n-t)\tilde{\varphi}_t\|_{N}
\]
for some $C,N>0$ independent of $t$. On the other hand on the support of $\chi(x_n-t)$ we have $|x_n-t|\leq C_1$ and 
\[
1+|x_n|\leq C_2(1+|t|)\leq C_3( 1+|t+\tau(x_n-t)|),
\]
for every $\tau\in[0,1]$, so that 
\begin{align*}
|x_n^{\alpha_n}\partial_x^\beta \tilde{\varphi}_t(x)|
&\leq \int_0^1 |x_n^{\alpha_n}| |\partial_x^\beta[\partial_{x_n}\varphi(x',t+\tau(x_n-t))]|\, d\tau
\\
&\leq C'\int_0^1 (1+|t+\tau(x_n-t)|)^{\alpha_n}|\partial_x^\beta[\partial_{x_n}\varphi(x',t+\tau(x_n-t))]|\, d\tau
\end{align*}
which gives 
\[
\|\chi(x_n-t)\tilde{\varphi}_t\|_{N}\leq C'' \|\varphi\|_{N+1}
\]
with constants $C'', N$ independent of $t$. This gives \eqref{eq1}.\par
The formula \eqref{eq0} is easily verified: for $\varphi\in\sn$,
\[
(x_n-t)u_t(\varphi)=u_t((x_n-t)\varphi)=v(\chi(x_n-t)\varphi)=\chi(x_n-t)v(\varphi).
\]
\end{proof}
We emphasize that the point in the above result is the control of the constants $C$ and $N$ with respect to $t$; in fact the existence of a temperate distribution solution $u_t$ for every fixed $t$ already follows from a variant of the arguments in \cite[page 127]{schwartz}.\par
We also need the following division result {\it with control of the support}. This result could probably be proved by extracting and combining several arguments disseminated in \cite{loja}, but we prefer to provide a self-contained and more accessible proof. 
As above we split the variable in $\rn$ as $x=(x',x_n)$. 
\begin{proposition}\label{prolo}
Let $B'=\{x'\in\mathbb{R}^{n-1}:\, |x'|<1\}$, $B=B'\times\mathbb{R}\subset\rn$  and $f:B'\to\mathbb{R}$ be a smooth function. Let $K:=\{x=(x',x_n)\in B:\ x_n\geq f(x')\}$. Suppose that 
\[
K_0:=\{x'\in B':\ f(x')=0\} =\{x'\in B':\ x_1=\ldots=x_k=0\}
\] for some $1\leq k\leq n-1$, and 
\begin{equation}\label{lo}
|f(x')|\geq C_0\,{\rm dist}(x',K_0)^N\quad x'\in B
\end{equation}
for some $C_0,N>0$. \par
Then, for every $v\in \mathcal{E}'(B)$ with ${\rm supp}\, v\subset K$ the equation
\begin{equation}\label{e1}
x_n u=v
\end{equation}
admits a solution $u\in\mathcal{E}'(B)$  with ${\rm supp}\, u\subset K$.
\end{proposition}
The condition \eqref{lo} is known as \L ojasiewicz' inequality and is automatically satisfied if $f$ is real-analytic (\cite[Theorem 17]{loja}).
\begin{proof}
As a preliminary remark, it is clear that we can limit ourselves to construct a solution $u\in\mathcal{D}'(B)$ with ${\rm supp}\, u\subset K$, because one can then multiple $u$ by a cut-off function, equal to $1$ in a neighborhhod of the support of $v$ and get another solution in $\mathcal{E}'(B)$, still satisfying ${\rm supp}\, u\subset K$. \par
Now, it is easy to see that a solution of \eqref{e1} is given by the distribution 
\[
C^\infty_c(B)\ni\varphi\mapsto v(\tilde{\varphi})
\]
where \begin{equation}\label{recall}
\tilde{\varphi}(x)=\int_0^1\partial_{x_n}\varphi(x', \tau x_n)\, d\tau.
\end{equation}
 More generally any distribution  of the form 
\begin{equation}\label{e2}
u(\varphi)= w\otimes \delta +v(\tilde{\varphi}),
\end{equation}
where $w$ is an arbitrary distribution in $B'$ and $\delta=\delta(x_n)$, is solution of \eqref{e1}. Hence we are reduced to prove that $w$ can be chosen so that ${\rm supp}\, u\subset K$, i.e. $u(\varphi)=0$ for every $\varphi$ supported in the {\it open} set $\Omega:=\{x\in B:\ x_n<f(x)\}$.  This condition, as we will see, forces the values of $w$ on the test functions $\varphi_1\in C^\infty_c(\Omega)$.
\par
We construct $w$ as follows.
 Fix once for all a function $\varphi_2(x_n)$ in $C^\infty_c(\mathbb{R})$, with $\varphi_2(0)=1$, supported in the interval $[-1,1]$. Let $\Omega'=\Omega\cap \{ x_n=0\}$. Let $\varphi_1(x')$ be any function in $C^\infty_c(\Omega')$ and $\epsilon>0$ such that
\[
{\rm dist}({\rm supp}\, \varphi_1, K_0)>\epsilon.
\]
Then we define
\[
w(\varphi_1):=-v(\tilde\varphi)
\]
where 
\begin{equation}\label{wd}
\varphi(x)=\varphi_1(x')\varphi_2(x_n/(C_0\epsilon^N)),
\end{equation}
 with the constants $C_0,N$ appearing in \eqref{lo}, and $\tilde{\varphi}$ is constructed from $\varphi$ as in \eqref{recall}:
 \begin{align}\label{recall2}
 \tilde{\varphi}(x)&=\frac{\phi_1(x')}{C_0\epsilon^N} \int_0^1 \varphi_2'(\tau x_n/(C_0\epsilon^N))\, d\tau\\
 &=\frac{\phi_1(x')}{x_n}\int_0^{x_n/(C_0\epsilon^N)} \varphi_2'(\tau)\, d\tau;\nonumber
 \end{align}
see Figure 1.
\begin{figure}[b]
\begin{tikzpicture}
\draw[->] (-1.5,0) -- (3,0) node[below]{$x'$};
\draw[thick] (0,0) -- (3,0);
\node at (3,0.7) {$\Omega'$};
\draw[->] (2.8,0.5) --(2.5,0.1);
\draw[->] (0,-0.5) -- (0,2.3) node[left]{$x_n$};
\draw[thick, domain=-0.8:1.8] plot (\x, {0.5*\x*(\x+1)}) node[below right] {$x_n=f(x')$};
\draw[dashed] (1,-0.3) -- (2,-0.3)-- (2,0.3) -- (1,0.3)-- (1,-0.3);
\node at (2,1.3) {$\Omega$};
\node at (0.8,2) {$K$};
\draw[<->, dashed] (0,-0.2)-- (0.9,-0.2);
\node at (0.5,-0.35) {$\epsilon$};
\end{tikzpicture}
\caption{The box contains the support of  $\varphi$ in \eqref{wd}.}. 
\end{figure}
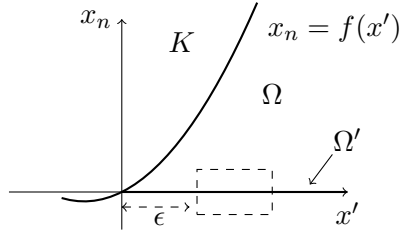

  Observe that the function $\varphi$ in \eqref{wd} is supported in ${\Omega}$ by \eqref{lo}, but this is not the case for $\tilde{\varphi}$.
 We now prove the following facts.\par\medskip
{\it 1) $w$ is well defined.} Let us verify that the definition of $w(\varphi_1)$ does not depend on the choice of $\epsilon$. Let ${\rm dist}({\rm supp}\, \varphi_1, K_0)>\epsilon>\epsilon'>0$; then the difference function
\[
\varphi(x)=\phi_1(x')[\varphi_2(x_n/(C_0\epsilon^N))- \varphi_2(x_n/(C_0{\epsilon'}^N))]
\]
is obviously still supported in $\Omega$ but, in addition, it vanishes at $x_n=0$, so that the corresponding function $\tilde{\varphi}$ in \eqref{recall} has compact support contained in $\Omega$. Since $v$ is supported in $K$ we have $v(\tilde\varphi)=0$. \par\medskip
{\it 2) $w\in \mathcal{D}'(\Omega')$}. This is easy to verify and is also a consequence of the next point. \par\medskip
{\it 3) $w$ extends to a distribution in $\mathcal{D}'(B')$}. It is sufficient to prove an estimate of the type 
\[
|w(\varphi_1)|\leq C\sup_{|\alpha|\leq M}\sup_{x'\in B'} |\partial_{x'}^\alpha \varphi_1(x')|
\]
for some constants $C,M>0$ and every  $\varphi_1\in C^\infty_c(\Omega')$. In fact by the Hahn-Banach theorem one can then extend the linear functional $w:C^\infty_c(\Omega')\to\mathbb{C}$, which is continuous when $C^\infty_c(\Omega')$ is endowed with the norm in the above right-hand side, to a functional on $C^\infty_c(B')$, continuous with respect to the same norm and therefore, a fortiori, for the usual topology of this space.\par
 Now, by the definition of $w$ and since $v$ has compact support in $K$ we have (cf.\ \cite[Theorem 2.3.10]{hormander})
\begin{equation}\label{e18}
|w(\varphi_1)|=|v(\tilde{\varphi})|\leq C\sup_{|\beta|\leq M}\sup_{x\in K\atop x_n\leq C_1} |\partial_x^\beta \tilde{\varphi}(x)|
\end{equation}
for some $C,C_1,M>0$. To estimate the last term we observe that in the expression for $\tilde{\varphi}$ in \eqref{recall2} the integral is in fact constant if $x_n\geq C_0\epsilon^N$. In particular this happens if $x\in K$ and $x'\in{\rm supp}\,\varphi_1$, because \eqref{lo} implies that for such $x$ it turns out \[
x_n\geq f(x)\geq C_0{\rm dist}(x',K_0)^N\geq C_0\epsilon^N.
\] Hence for $x\in K$, $x_n\leq C_1$ we have  
\begin{align*}
\sup_{|\beta|\leq M} |\partial_x^\beta \tilde{\varphi}(x)|&\leq C\sup_{|\alpha|\leq M}|\partial_{x'}^\alpha {\varphi}_1(x')|\cdot x_n^{-M}\\
&\leq C\sup_{|\alpha|\leq M}|\partial_{x'}^\alpha {\varphi}_1(x')|\cdot C_0^{-M} {\rm dist}(x',K_0)^{-N M}.
\end{align*}
On the other hand we have \[
{\rm dist}(x',K_0)=(|x_{1}|^2+\ldots+|x_{k}|^2)^{1/2},
\] so that a Taylor expansion (with remainder of order $NM$) of $\partial_{x'}^\alpha\varphi_1$ with respect to $x_{1},\ldots, x_{k}$ (taking into account that $\varphi_1$ vanishes in a neighborhood of $K_0$) gives 
\[
\sup_{|\beta|\leq M}\sup_{x\in K\atop x_n\leq C_1} |\partial_x^\beta \tilde{\varphi}(x)|\leq  C'\sup_{|\alpha|\leq M+NM}\sup_{x'\in B'}|\partial_{x'}^\alpha {\varphi}_1(x')|.
\]
This together with \eqref{e18} gives the desired conclusion.\par\medskip
{\it 4) With this choice of $w$ in \eqref{e2}, we have ${\rm supp}\, u\subset K$}. Let $\varphi\in C^\infty_c(\Omega)$, and write
\[
\varphi(x)= [\varphi(x)-\varphi(x',0)\varphi_2(x_n/(C_0\epsilon^N))]+\varphi(x',0)\varphi_2(x_n/(C_0\epsilon^N)).
\]
where we choose $\epsilon<{\rm dist}\, ({\rm supp}\, \varphi(\cdot,0),K_0)$.\par Now the distribution $u$ vanishes when applied to the second term of this sum just by the definition of $w$ (with $\varphi(x',0)$ playing the role of $\varphi_1(x')$). On the other hand the function $\varphi(x)-\varphi(x',0)\varphi_2(x_n/(C_0\epsilon^N))$ is supported in $\Omega$ {\it and vanishes at $x_n=0$}, so that one sees from the definition of $u$ in \eqref{e2} that its pairing with $u$ is $0$.
\end{proof}

\subsection{Changes of coordinates for temperate distributions}\label{changevar}
In the sequel we will perform changes of coordinates which preserve Schwartz functions and temperate distributions in the following sense.\par
Let $\phi$ be a smooth diffeomorphism of the semispace $\{x_1>0\}\subset\rn$ into itself. Suppose that $\phi$ is positively homogeneous for some {\it positive} order, say $r>0$, i.e. $\phi(\lambda x)=\lambda^r\phi(x)$ for every $x\in\rn$ with $x_1>0$, $\lambda>0$. It follows that the image of every truncated cone
\[
U=\{x\in\rn:\ x_1\geq \epsilon|x|,\ |x|\geq \epsilon\},
\]
with $0<\epsilon\leq 1$ (see Figure 2 below) is contained in another truncated cone of the same type. In fact, if $y=\phi(x)$, then for some $\epsilon'>0$ we have $|y|\geq \epsilon'$ when $x\in U$ and $|x|=\epsilon$ by compactness, and therefore for every $x\in U$ by homogeneity. The same argument implies that $y_1/|y|\geq \epsilon'>0$ for $x\in U$.\par
\begin{figure}[b]
\begin{tikzpicture}
\draw[->] (1.7,1) -- (5,1) node[above]{$x_1$};
\draw[->] (2,-0.3) -- (2,2.3) node[right]{$x'$};
\draw[thick] (2.85,0.5) -- ++(-30:2cm)
      (2.85,1.5) -- ++(30:2cm) node[below]{$U$};
      (2,1) -- ++(30:1cm);
\draw[thick] ([shift=(-30:1cm)]2,1) arc (-30:30:1cm);
\end{tikzpicture}
\caption{A truncated cone in the semispace $x_1>0$.}
\end{figure}
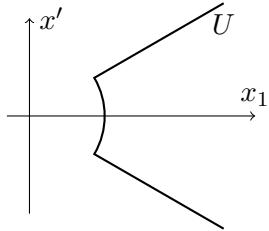

 Moreover the same applies to the inverse function $\phi^{-1}(x)$, which will be homogeneous of degree $1/r>0$.\par
 Now, let $\sn_{cone}$, $\spn_{cone}$ be the spaces of Schwartz functions and temperate distributions in $\rn$, respectively, with support contained in {\it some} truncated cone, as above. Then $\phi$ induces bijections \[
 \phi_\ast:\sn_{cone}\to \sn_{cone},\quad \phi_\ast:\spn_{cone}\to \spn_{cone}\]
  defined as follows.\par If $\psi\in\sn_{cone}$ we define $\phi_\ast\psi(x)= \psi(\phi^{-1}(x))$ for $x_1>0$ and $=0$ otherwise, and it is easy to see that $\phi_\ast\psi\in\sn_{cone}$ using the homogeneity of $\phi^{-1}$ and the support condition on $\psi$.\par If $u\in \spn_{cone}$ we define the distribution $\phi_\ast u$ by 
 \[
 \phi_\ast u(\varphi)=u(\chi \cdot \varphi\circ\phi |\det\phi'|),
 \] for every $\varphi\in\sn$, where $\chi:\rn\to[0,1]$ is a smooth function, positively homogeneous of degree $0$ for large $|x|$, $\chi(x)=1$ on a truncated cone slightly larger than one containing the support of $u$ and $\chi$ is supported in a truncated cone. It is easy to see that $\chi(x)\varphi(\phi(x))|\det\phi'(x)|$ is a function in $\sn_{cone}$ because on every truncated cone the Jacobian determinant $|{\rm det}\,\phi' |$ is smooth and has an at most polynomial growth, together with its derivatives.\par
The maps $\phi_\ast:\sn_{cone}\to \sn_{cone}$  and $\phi_\ast:\spn_{cone}\to \spn_{cone}$ are bijections, the inverses being given by $(\phi^{-1})_\ast$ on $\sn_{cone}$ and $\spn_{cone}$. \par
Observe that, more generally, if $\psi$ is a smooth function with an at most polynomial growth together with its derivatives, we can similarly define $\phi_\ast\psi(x)= \psi(\phi^{-1}(x))$ for $x_1>0$ 
and we have the formula 
\begin{equation}\label{product}
\phi_\ast (\psi u)=\phi_\ast \psi\,  \phi_\ast u
\end{equation}
if $u\in \spn_{cone}$ (the formula makes sense even if $\phi_\ast \psi$ is defined only for $x_1>0$, because $u$ and therefore $\phi_\ast u$ is supported in a truncated cone).

\section{Born--Jordan Pseudodifferential Operators}

\subsection{The Weyl correspondence}

We recall that the Heisenberg operator $\widehat{T}(z_{0})$ is defined, for
$z_{0}=(x_{0},p_{0})$, by
\begin{equation}
\widehat{T}(z_{0})\psi(x)=e^{\tfrac{i}{\hslash}(p_{0}x-\tfrac{1}{2}p_{0}%
x_{0})}\psi(x-x_{0})\label{HW}%
\end{equation}
and the Grossmann--Royer operator is defined by%
\begin{equation}
\widehat{T}_{\text{GR}}(z_{0})\psi(x)=e^{\frac{2i}{\hslash}p_{0}(x-x_{0})}%
\psi(2x_{0}-x).\label{GR}%
\end{equation}
Both are unitary operators on $L^{2}(\mathbb{R}^{n})$, and $\widehat{T}%
_{\text{GR}}(z_{0})$ is an involution: 
\[
\widehat{T}_{\text{GR}}(z_{0}%
)\widehat{T}_{\text{GR}}(z_{0})=I.
\] The two following important formulas hold:%
\begin{equation}
\widehat{T}_{\text{GR}}(z_{0})=\widehat{T}(z_{0})R^{\vee}\widehat{T}%
(z_{0})^{-1}\label{groroy2}%
\end{equation}
where $R^{\vee}=\widehat{T}_{\text{GR}}(0)$ is the reflection operator:
$R^{\vee}\psi(x)=\psi(-x)$, and%
\begin{equation}
\widehat{T}_{\text{GR}}(z_{0})\psi(x)=2^{-n}F_{\sigma}[\widehat{T}(\cdot
)\psi(x)](-z_{0})\label{defgr}%
\end{equation}
where $F_{\sigma}$ is the symplectic Fourier transform (see \cite{Birkbis}).

Let $a\in\mathcal{S}^{\prime}(\mathbb{R}^{2n})$ (hereafter to be called a
\textit{symbol}). The Weyl operator $\widehat{A}_{\mathrm{W}}={\rm Op}_{\mathrm{W}%
}(a)$ is the operator $\mathcal{S}(\mathbb{R}^{n})\longrightarrow
\mathcal{S}^{\prime}(\mathbb{R}^{n})$ defined by
\begin{equation}
\widehat{A}_{\mathrm{W}}\psi=\left(  \tfrac{1}{\pi\hslash}\right)  ^{n}\int_{\rnn}
a(z_{0})\widehat{T}_{\text{GR}}(z_{0})\psi dz_{0}\label{aweyl1}%
\end{equation}
which is equivalent, using (\ref{defgr}), to
\begin{equation}
\widehat{A}_{\mathrm{W}}\psi=\left(  \tfrac{1}{2\pi\hslash}\right)  ^{n}\int_{\rnn}
a_{\sigma}(z_{0})\widehat{T}(z_{0})\psi dz_{0}\label{aweyl2}%
\end{equation}
where $a_{\sigma}=F_{\sigma}a$ ($a_{\sigma}$ is sometimes called the covariant
Weyl symbol of $\widehat{A}_{\mathrm{W}}$). These integrals are meant in the weak sense. It is important to recall (see e.g. \cite[Chapter 10]{Birkbis}) the following result. \par\medskip
{\it Every linear continuous operator $\widehat{A}:\sn\to\spn$ can be written in a unique way in Weyl form, i.e. there exists a unique symbol $a\in\spnn$ such that $\widehat{A}={\rm Op}_{\mathrm{W}}(a)$}.\par\medskip 
 Notice that when 
$a\in \mathcal{S}(\mathbb{R}^{2n})$ formula (\ref{aweyl1}) can be rewritten in the
familiar form%
\begin{equation}
\widehat{A}_{\mathrm{W}}\psi(x)=\left(  \tfrac{1}{2\pi\hslash}\right)  ^{n}\int_{\rnn}
e^{\frac{i}{\hslash}p(x-y)}a(\tfrac{1}{2}(x+y),p)\psi(y)dpdy.\label{aweyl3}%
\end{equation}

\subsection{Born--Jordan operators}

The Born--Jordan operator $\widehat{A}_{\mathrm{BJ}}={\rm Op}_{\mathrm{BJ}}(a)$ is
constructed as follows: one first defines the Shubin $\tau$-operator
$\widehat{A}_{\tau}=\operatorname*{Op}_{\tau}(a)$ by
\begin{equation}
\widehat{A}_{\tau}\psi=\int_{\rnn} a_{\sigma}(z_{0})\widehat{T}_{\tau}(z_{0})\psi
dz_{0}\label{atauharmonic}%
\end{equation}
where $\widehat{T}_{\tau}(z)$ is the unitary operator on $L^{2}(\mathbb{R}%
^{n})$ defined by
\begin{equation}
\widehat{T}_{\tau}(z_{0})=e^{\frac{i}{\hslash}(\tau-\frac{1}{2})p_{0}x_{0}%
}\widehat{T}(z_{0}).\label{tatau}%
\end{equation}
One thereafter defines%
\[
\widehat{A}_{\mathrm{BJ}}\psi=\int_{0}^{1}\widehat{A}_{\tau}\psi d\tau.
\]
Using the obvious formula%
\[
\Theta
(z_{0}):=\int_{0}^{1}e^{\frac{i}{\hslash}(\tau-\frac{1}{2})p_{0}x_{0}}d\tau=\begin{cases}
\displaystyle\frac{\sin(p_{0}x_{0}/2\hslash)}{p_{0}x_{0}/2\hslash}&\textrm{for}
\ p_{0}x_{0}\neq0\\
1&\textrm{for}\ p_{0}x_{0}=0
\end{cases}
\]
one thus has%
\[
\widehat{A}_{\mathrm{BJ}}\psi=\int a_{\sigma}(z_{0})\widehat{T}_{\mathrm{BJ}%
}(z_{0})\psi dz_{0}%
\]
where $\widehat{T}_{\mathrm{BJ}}(z_{0})=\Theta(z_{0})\widehat{T}(z_{0})$.

It follows that $\widehat{A}_{\mathrm{BJ}}(a)$ is the Weyl operator with
covariant symbol 
\begin{equation}\label{eq11}
(a_{\mathrm{W}})_{\sigma}=\Theta a_\sigma.
\end{equation} 
The Weyl symbol of
$\widehat{A}_{\mathrm{BJ}}(a)$ is thus (taking the symplectic Fourier transform)
\begin{equation}
a_{\mathrm{W}}=\left(  \tfrac{1}{2\pi\hslash}\right)  ^{n}a\ast\Theta_{\sigma
}.\label{aw}%
\end{equation}
Conversely, assume that  $\widehat{A}={\rm Op}_{\mathrm{W}}(a_{\mathrm{W}})$. Then $\widehat{A}={\rm Op}_{\mathrm{BJ}}(a)$ provided that $a$ satisfies
$
 (a_{\mathrm{W}})_{\sigma}=\Theta a_\sigma .
$

\section{The case of polynomial symbols}\label{polynomial}

In this section we work in dimension $n=1$ (for simplicity) and we study in detail the Born-Jordan quantization of  polynomial symbols.\par
 Let $\mathbb{C}[x,p]$ the ring of polynomials generated by the two
indeterminates $x$ and $p$: it consists of all finite formal sums
$a=\sum_{r,s}\alpha_{rs}x^{r}p^{s}$ where the coefficients $\alpha_{rs}$ are
complex numbers; it is assumed that $x^{r}p^{s}=p^{s}x^{r}$ hence
$\mathbb{C}[x,p]$ is a commutative ring. We identify $\mathbb{C}[x,p]$ with
the corresponding ring of polynomial functions. We will denote by
$\mathbb{C}[\widehat{x},\widehat{p}]$ the Weyl algebra; it is the universal
enveloping algebra of the Heisenberg Lie algebra \cite{coutino,tosiek}, and is
realized as the non-commutative unital algebra generated by $\widehat{x}$ and
$\widehat{p}$, two indeterminates satisfying the commutation relation
\begin{equation}
\lbrack\widehat{x},\widehat{p}]=\widehat{x}\widehat{p}-\widehat{p}\widehat
{x}=i\hbar\mathrm{1}\label{CCR}%
\end{equation}
where $\mathrm{1}$ is the unit of $\mathbb{C}[\widehat{x},\widehat{p}]$; we
will abuse notation by writing $i\hbar\mathrm{1}\equiv i\hbar$ . We choose
here for $\widehat{x}$ the operator of multiplication by $x$ and $\widehat
{p}=-i\hbar\partial_{x}$. Each $\widehat{A}\in\mathbb{C}[\widehat{x}%
,\widehat{p}]$ can be written (uniquely) as a finite sum of terms $\widehat
{x}^{r}\widehat{p}^{s}$. The Weyl and Born--Jordan quantizations are linear
mappings $\mathbb{C}[x,p]\longrightarrow\mathbb{C}[\widehat{x},\widehat{p}]$;
this immediately follows by the linearity of quantization from the formulas%
\begin{align}
\operatorname*{Op}\nolimits_{\mathrm{W}}(x^{r}p^{s}) &  =\sum_{\ell=0}%
^{\min(r,s)}(-i\hbar)^{\ell}\binom{s}{\ell}\binom{r}{\ell}\frac{\ell!}%
{2^{\ell}}\widehat{x}^{r-\ell}\widehat{p}^{s-\ell}.\label{Kerner3}\\
\operatorname*{Op}\nolimits_{\mathrm{BJ}}(x^{r}p^{s}) &  =\sum_{\ell=0}%
^{\min(r,s)}(-i\hbar)^{\ell}\binom{s}{\ell}\binom{r}{\ell}\frac{\ell!}{\ell
+1}\widehat{x}^{r-\ell}\widehat{p}^{s-\ell}\label{Kerner2}%
\end{align}
which easily follow from (\ref{weylmono1}) and (\ref{bjmono1}) by repeated use
of the commutation relation (\ref{CCR}). 

\begin{remark}
It follows from the formulas above that $\operatorname*{Op}%
\nolimits_{\mathrm{W}}(x^{r}p^{s})\neq\operatorname*{Op}\nolimits_{\mathrm{BJ}%
}(p^{s}x^{r})$ as soon as $r\geq2$ and $s\geq2$; for instance%
\begin{align}
\operatorname*{Op}\nolimits_{\mathrm{W}}(x^{2}p^{2}) &  =\widehat{x}%
^{2}\widehat{p}^{2}-2i\hbar\widehat{x}\widehat{p}-\tfrac{1}{2}\hbar
^{2}\label{opbj22}\\
\operatorname*{Op}\nolimits_{\mathrm{BJ}}(x^{2}p^{2}) &  =\widehat{x}%
^{2}\widehat{p}^{2}-2i\hbar\widehat{x}\widehat{p}-\tfrac{2}{3}\hbar
^{2}.\label{opw22}%
\end{align}

\end{remark}

It is well-known that Weyl quantization is an isomorphism of vector spaces $\mathbb{C}%
[x,p]\longrightarrow\mathbb{C}[\widehat{x},\widehat{p}]$ (Cohen
\cite{Cohen1,Cohenbook}); explicit formulas for the inverse can be found in
the literature \cite{doga15,tosiek,prza}, but they are very complicated and we
will not reproduce them here. It follows from this property that:

\begin{proposition}
The Born--Jordan quantization of polynomials is an isomorphism of vector spaces
\[
\operatorname*{Op}\nolimits_{\mathrm{BJ}}:\mathbb{C}[x,p]\longrightarrow
\mathbb{C}[\widehat{x},\widehat{p}].
\]

\end{proposition}

\begin{proof}
We begin by noting that the Weyl transform being an isomorphism $\mathbb{C}%
[x,p]\longrightarrow\mathbb{C}[\widehat{x},\widehat{p}]$, every $\widehat
{A}\in\mathbb{C}[\widehat{x},\widehat{p}]$ can be written $\widehat
{A}=\operatorname*{Op}\nolimits_{\mathrm{W}}(b)$ for a unique $b\in
\mathbb{C}[x,p]$. This allows us to define an endomorphism $T$ of
$\mathbb{C}[\widehat{x},\widehat{p}]$ by
\[
T(\operatorname*{Op}\nolimits_{\mathrm{W}}(a))=\operatorname*{Op}%
\nolimits_{\mathrm{BJ}}(a)=\operatorname*{Op}\nolimits_{\mathrm{W}}(a\ast \Theta_\sigma).
\]
Let us show that $T$ is bijective; this will prove our assertion. First, it is
clear that $T$ is injective: if $T(\operatorname*{Op}\nolimits_{\mathrm{W}%
}(a))=0$ then $\Theta a_\sigma$ is zero as a distribution, but this is only
possible if $a=0$ since $a$ is a polynomial, so that $a_\sigma$ is supported at $0$, and $\Theta$ does not vanish in a neighborhood of $0$.\par
 Let us now
prove that $T$ is surjective. Since $\mathbb{C}[x,p]$ is spanned by the
monomials $b(x,p)=x^{r}p^{s}$ it is sufficient to show that there exists
$a\in\mathbb{C}[x,p]$ such that $F_{\sigma}b=\Theta F_{\sigma}a$; since
$F_{\sigma}a(z)=Fa(Jz)$ where $F$ is the usual ($\hbar$-dependent) Fourier
transform on $\mathbb{R}^{2n}$ and $\Theta(Jz)=\Theta(z)$, this is equivalent
to the equation $Fb(z)= \Theta(z) Fa(z)$. Since
\[
Fb(z)=F(x^{r}\otimes p^{s})=2\pi \hslash(i\hbar)^{r+s}\delta_{x}^{(r)}\otimes\delta_{p}^{(s)}%
\]
the Fourier transform of $a$ is then given by
\[
Fa(x,p)= 2\pi \hslash (i\hbar)^{r+s}\Theta
(x,p)^{-1} \delta_{x}^{(r)}\otimes\delta_{p}^{(s)}.
\]
Using the Laurent series expansion of $1/\sin x$ we have
\[
\Theta(x,p)^{-1}=\sum_{k=0}^{\infty}a_{k}(2\hbar)^{-2k}x^{2k}p^{2k}%
\]
where the coefficients are expressed in terms of the Bernoulli numbers $B_{n}$
by
\[
a_{k}=\frac{(-1)^{k-1}(2^{2k}-2)B_{2k}}{(2k)!};
\]
the series is convergent in the open set $|xp|<2\hbar\pi$. It follows that%
\begin{align*}
Fa(x,p) &  =2\pi\hslash\sum_{k=0}^{\infty}a_{k}(2\hbar)^{-2k}(x^{2k}\delta_{x}%
^{(r)})(p^{2k}\delta_{p}^{(s)})\\
&  =2\pi\hslash\sum_{k=0}^{n_{r,s}}a_{k}\frac{(2\hbar)^{-2k}r!s!}{(r-2k)!(s-2k)!}%
\delta_{x}^{(r-2k)}\delta_{p}^{(s-2k)}
\end{align*}
with $n_{r,s}=[\frac{1}{2}\min(r,s)]$ ($[\cdot]$ denoting the integer part). Setting
\[
b_{k}=a_{k}\frac{(2\hbar)^{-2k}r!s!}{(r-2k)!(s-2k)!}%
\]
and noting that
\[
2\pi\hslash\delta_{x}^{(r-2k)}\delta_{p}^{(s-2k)}=(i\hbar)^{-(r+s-4k)}F(x^{r-2k}p^{s-2k})
\]
we have
\[
a(x,p)=\sum_{k=0}^{n_{r,s}}b_{k}(i\hbar)^{-(r+s-4k)}x^{r-2k}p^{s-2k}%
\]
hence $a\in\mathbb{C}[x,p]$.
\end{proof}

\section{Invertibility of the Born-Jordan quantization}
In this section we investigate the injectivity and surjectivity of the map $\spnn\to\spnn$ given by
\begin{equation}\label{map}
a\longmapsto \left(\tfrac{1}{\tph}\right)^n a \ast \Theta_\sigma
\end{equation}
namely the map which gives the Weyl symbol of an operator with Born-Jordan symbol $a$, cf.\ \eqref{aw}.\par
We have the following result.
\begin{theorem}\label{mainteo0}
The equation 
\[
\left(\tfrac{1}{\tph}\right)^n a \ast \Theta_\sigma=b.
\]
admits a solution $a\in\spnn$, for every $b\in\spnn$.
\end{theorem}
\begin{proof}
Taking the symplectic Fourier transform we are reduced to prove that the equation 
\[
 \Theta a=b
\]
admits at least a solution $a\in \spnn$, for every $b\in\spnn$. 
This is a problem of division of {\it temperate} distributions. We provide two proofs. \par\medskip
{\it First proof.}\par

We localize the problem by considering a finite and smooth partition of unity $\psi_0(z)$, $\psi^\pm_j(z)$, $j=1,\ldots,2n$, in phase space, where $\psi_0$ has support in a ball $|z|\leq r$, with $r<\sqrt{4\pi\hslash}$, and $\psi^+_j$ for $j=1,\ldots,2n$, is supported in a truncated cone (cf.\  Section \ref{changevar}) of the type
\[
U^+_j=\{z\in\rnn:\ z_j\geq \epsilon |z|,\ |z|\geq \epsilon\}
\]
contained in the semispace $z_j>0$, and similarly $\psi^-_j$ is supported in a similar truncated cone contained in the semispace $z_j<0$, with $\psi^\pm_j$ homogeneous of degree $0$ for large $z$ (it is easy to see that such a partition of unity can be constructed if $\epsilon$ is small enough, e.g.\ if $\epsilon<r/2$ and $\epsilon<1/(2\sqrt{n})$).\par
  It is clear that if $a_0$ and $a_j^{\pm}$, $j=1,\ldots,2n$, solve in $\spnn$ the equations $ \Theta a_0=\psi_0 b$ and $\Theta a_j^{\pm}= \psi_j^{\pm}b$, then $a:= a_0+\sum_{j=1}^n (a^+_j+a_j^-)$ solves $\Theta a=b$.\par
The equation $\Theta a_0=\psi_0 b$ is easily solved as in the proof of Proposition \ref{pro6}, using the fact that $\psi_0 b$ is supported in a closed ball where $\Theta\not=0$. \par
Let us now solve the equation
\[
\Theta a=b
\]
where $b\in\spnn$ is supported in a truncated cone in the semispace, say, $z_1>0$ in phase space. We look for $a\in\spnn$ supported in a truncated cone as well. We apply the following algebraic change of variables in phase space:
\begin{align*}
y_1&=z_1^2,\ y_2=z_1z_2, \ldots,\ y_n= z_1z_n,\\
 y_{n+1}&=z_1z_{n+2},\ldots,\ y_{2n-1}=z_1 z_{2n},\ y_{2n}=x p=\sum_{j=1}^n z_j z_{j+n}
\end{align*}
where $z=(x,p)$. \par It is easy to check that the map $z\longmapsto y$ is a diffeomorphism of the semispace $z_1>0$ into itself ($y_1>0$)\footnote{
The inverse change of variables is given by 
\begin{align*}
z_1&=\sqrt{y_1},\ z_2=y_2/\sqrt{y_1},\ldots,\ z_n=y_n/\sqrt{y_1},\\
z_{n+1}&=\Big(y_{2n}-\sum_{j=2}^n y_j y_{n+j}/{y_1}\Big)/\sqrt{y_1},\ z_{n+2}=y_{n+1}/\sqrt{y_1},\ldots, z_{2n}=y_{2n-1}/\sqrt{y_1}.
\end{align*}
}, and moreover it is homogeneous of degree 2. We now apply the remarks in Section \ref{changevar} (where the dimension of the space is now $2n$), in particular \eqref{product}, and we are reduced to solve (in the new coordinates) the equation
\begin{equation}\label{eq4}
\frac{\sin (y_{2n}/2\hslash)}{y_{2n}/2\hslash} a=b
\end{equation}
where $b\in\spnn$ is supported in a truncated cone in the semispace $y_1>0$ and we look for $a\in\spnn$ similarly supported in a truncated cone in the same semispace. \par 
We now consider a partition of unity in $\mathbb{R}$ obtained by translation of a fixed function, of the type $\chi(y_{2n}-2\pi k\hslash)$, $k\in\mathbb{Z}$, where $\chi\in C^\infty_c(\mathbb{R})$ is a fixed function supported in the interval $[-(3/2)\pi\hslash,(3/2)\pi\hslash]$. Observe that on the support of $\chi(y_{2n}-2\pi k\hslash)$ the function $\frac{\sin (y_{2n}/2\hslash)}{y_{2n}/2\hslash}$ has only a simple zero at $y_{2n}=2\pi k\hslash$, for $k\not=0$, whereas it does not vanish on the support of $\chi$ (case $k=0$).  \par
We now solve, for every $k\in\mathbb{Z}$, the equation 
\begin{equation}\label{eq5}
\frac{\sin (y_{2n}/2\hslash)}{y_{2n}/2\hslash} a_k=\chi(y_{2n}-2\pi k\hslash)b.
\end{equation}
We suppose $k\not=0$, the case $k=0$ being easier. Since the function \[
\frac{\sin (y_{2n}/2\hslash)}{(y_{2n}-2k\pi \hslash)y_{2n}/2\hslash}
\]
 is smooth and does not vanishes on the support of $\chi(y_{2n}-2\pi k\hslash)$ it is sufficient to solve the equation 
\[
 (y_{2n}-2k\pi \hslash) a_k=\chi(y_{2n}-2\pi k\hslash)\Big[\frac{(y_{2n}-2k\pi \hslash)y_{2n}/2\hslash}{\sin (y_{2n}/2\hslash)} b\Big].
\]
 Observe, that the function $\frac{y_{2n}-2k\pi \hslash}{\sin (y_{2n}/2\hslash)}$ has derivatives uniformly bounded with respect to $k$ (in fact we have $\sin(y_{2n}/2\hslash)=\pm\sin ((y_{2n}-2\pi k\hslash)/2\hslash )$, according to the parity of $k$, and the ``{\rm sinc}'' function has bounded derivatives of any order). It follows from this remark and Proposition \ref{pro2} that there exists a solution $a_k\in \spnn$ of the above equation, satisfying the estimate
\[
|a_k(\varphi)|\leq C\|\varphi\|_N \quad\forall\varphi\in\snn
\]
for some constants $C,N>0$ independent of $k$.\par Moreover we can suppose that all the $a_k$'s are supported in a fixed truncated cone (by multiplying by a cut-off function in phase space with the same property as $\psi^+_1(z)$ above, and $=1$ on a truncated cone slightly larger than one containing the support of $b$). We can also multiply $a_k$ by a cut-off function $\tilde{\chi}(y_{2n}-2\pi k\hslash)$, where $\tilde{\chi}\in C^\infty_c(\mathbb{R})$, and $\tilde{\chi}=1$ in a neighborhood of the support of $\chi$, and obtain new solutions
\[
\tilde{a}_k:= \tilde{\chi}(y_{2n}-2\pi k\hslash)a_k
\]
to \eqref{eq5}, satisfying 
\begin{equation}\label{stima}
|\tilde{a}_k(\varphi)|\leq C\|\tilde{\chi}(y_{2n}-2\pi k\hslash) \varphi\|_N
\end{equation}
for every $\varphi\in\snn$; see Figure 3 below.
 \begin{figure}[b]
\begin{tikzpicture}
\draw[->] (1.7,1) -- (7,1) node[above]{$y_1$};
\draw[->] (2,-0.3) -- (2,3) node[right]{$y_{2}$};
\draw[thick] (2.85,0.5) -- ++(-30:2cm)
      (2.85,1.5) -- ++(30:3cm);
 \draw[thick] (3.7,2) -- (6.7,2);
  \draw[dashed] (2,2) -- (3.7,2);
\draw[] (4.2,2.3) -- (7.2,2.3);
\draw[] (3.2,1.7) -- (6.3,1.7);
\draw[] (3.5,1.7) -- ++(30:1.2cm);
\draw[] (3.8,1.7) -- ++(30:1.2cm);
\draw[] (4.1,1.7) -- ++(30:1.2cm);
\draw[] (4.4,1.7) -- ++(30:1.2cm);
\draw[] (4.7,1.7) -- ++(30:1.2cm);
\draw[] (5,1.7) -- ++(30:1.2cm);
\draw[] (5.3,1.7) -- ++(30:1.2cm);
\draw[] (5.6,1.7) -- ++(30:1.2cm);
\draw[] (5.9,1.7) -- ++(30:1.2cm);
  \node at (1.4,2) {$2\pi k\hslash$};
\draw[thick] ([shift=(-30:1cm)]2,1) arc (-30:30:1cm);
\end{tikzpicture}
\caption{The distribution $\tilde{a}_k$ is supported in the strip ($n=1$ in this figure).}
\end{figure}
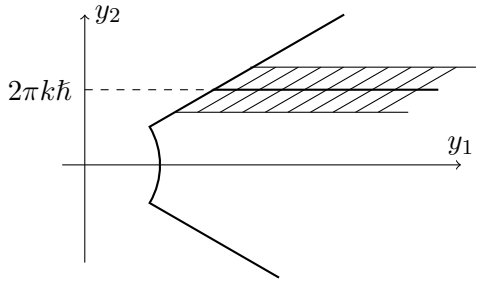\par
Now we claim that $a:=\sum_{k\in\mathbb{Z}}\tilde{a}_k$ solves \eqref{eq4}. We have only to check that the series converges in $\spnn$. Let us verify that, given $\varphi\in\snn$,  the series 
\[
\sum_{n\in\mathbb{Z}}\tilde{a}_k(\varphi)
\]
converges absolutely.\par On the support of $\tilde{\chi}(y_{2n}-2\pi k\hslash)$ we have $ 1+|y|\geq 1+|y_{2n}|\geq C(1+|k|)$, and $|y^\alpha\partial_y^\beta\varphi(y)|\leq C'_{N'}(1+|k|)^{-N'}$ {\it for every} $N'$ and $\alpha,\beta\in\mathbb{N}^{2n}$. Hence by \eqref{stima} we obtain
\[
|\tilde{a}_k(\varphi)|\leq C''_{N'}(1+|k|)^{-N'}, 
\]
and it is sufficient to take $N'=2$ for the above series to converge absolutely.  \par\medskip
{\it Second proof.}\par
In \cite{hormander0} it was proved that the equation $Pu=v$ is always solvable in $\spn$ if $P$ is a non identically zero polynomial. As observed there (page 556), that proof continues to hold if the polynomial $P$ is replaced by a smooth function which satisfies the estimates (4.3) and (4.10)  in that paper. Here we are interesting in the division by the function $\Theta(z)$, which has only simple zeros, and those estimates read
\begin{equation}\label{ho1}
|\Theta(z)|\geq C{\rm dist}(z,Z)^{\mu'} (1+|z|)^{-\mu''}\quad\forall z\in\rnn
\end{equation}
and 
\begin{equation}\label{ho2}
|\nabla\Theta(z)|\geq C (1+|z|)^{-\mu''}\quad \forall z\in Z
\end{equation}
for some $C,\mu',\mu''>0$, where $Z=\{z\in\rnn:\,\Theta(z)=0\}$. \par
To check that these estimates are satisfied, let ${\rm sinc} (t)=\sin t/t$ (${\rm sinc}(0)=1)$, so that $\Theta(z)=\Theta(x,p)={\rm sinc}\, (xp/2\hslash)$, $z=(x,p)$. Observe that at points $t$ where ${\rm sinc}\, (t/2\hslash)=0$ we have
\[
|\frac{d}{dt}{\rm sinc}\, (t/2\hslash)|=1/|t|
\]
so that 
\[ 
|\nabla \Theta(z)|=\frac{|(x,p)|}{|xp|}\geq\frac{2}{|z|}\quad \forall z\in Z
\]
which implies \eqref{ho2} with $\mu''=1$.\par
Concerning \eqref{ho1} observe first of all that, setting $Z_0=\{2\pi k\hslash:\ k\in\mathbb{Z}, k\not=0\}\subset\mathbb{R}$ we have, for $|t|>\pi\hslash$, 
\[
|{\rm sinc}(t/2\hslash)|=\frac{2\hslash}{|t|}|\sin(t/2\hslash)| \geq \frac{2\hslash}{|t|}\frac{1}{\pi\hslash}{\rm dist}(t,Z_0)= \frac{2}{\pi|t|}{\rm dist}(t,Z_0)
\]
whereas if $|t|\leq \pi\hslash$, 
\[
|{\rm sinc}(t/2\hslash)|\geq \frac{2}{\pi}\geq \frac{1}{\pi^2\hslash}{\rm dist}(t,Z_0).
\]
In both cases we have 
\begin{equation}\label{ho3}
|{\rm sinc}(t/2\hslash)|\geq C_0 (1+|t|)^{-1}{\rm dist}(t,Z_0)
\end{equation}
for some $C_0>0$.\par
Now, \eqref{ho1} is clearly satisfied in a neighborhood of $0$, so that it is sufficient to prove it in any truncated cone contained in the semispaces $z_j>0$ or $z_j<0$, $j=1,\ldots,2n$. Consider for example a truncaded cone $U$ where $z_1>0$. We perform the change of coordinates $y=y(z)$ in this semispace, exactly as in the previous proof, and we observe that by \eqref{ho3} we have
\[
|{\rm sinc}(y_{2n}/2\hslash)|\geq C_0(1+|y_{2n}|)^{-1}|y_{2n}-\overline{y}_{2n}|\geq C_0(1+|y|)^{-1}|y_{2n}-\overline{y}_{2n}|
\]
where $\overline{y}_{2n}\in Z_0$ is such that $|y_{2n}-\overline{y}_{2n}|={\rm dist}(y_{2n},Z_0)$. Now, for $z\in U$ we have $0<\epsilon \leq |y|\leq C |z|^2$ and moreover the inverse map $z=z(y)$ is Lipschitz in any truncated cone $U' \supset z(U)$, because the derivatives $\partial z_j/\partial y_k$ are positively homogeneous of degree $-1/2<0$, and therefore bounded in $U'$. Using these facts and setting $\overline{y}=(y_1,\ldots,y_{2n-1},\overline{y}_{2n})$, $\overline{z}=z(\overline{y})\in Z$, we conclude that for every $z\in U$,
\begin{align*}
|\Theta(z)|=|{\rm sinc}(y_{2n}/2\hslash)|&\geq C_0(1+|y|)^{-1}|y-\overline{y}|
\\
&\geq C(1+|z|)^{-2} |z-\overline{z}|,\\
&\geq C(1+|z|)^{-2} {\rm dist}(z,Z).
\end{align*}
This concludes the proof.
\end{proof}

From the previous theorem,  we obtain at once the following result. 
\begin{theorem}\label{mainteo1}
For every $b\in\spnn$ there exists a symbol $a\in\spnn$ such that $\operatorname*{Op}\nolimits_{\mathrm{BJ}}(a)=\operatorname*{Op}\nolimits_{\mathrm{W}}(b)$. \par
Hence, every linear continuous operator $\widehat{A}:\sn\to\spn$ can be written in Born-Jordan form, i.e.\ there exists a symbol $a\in\spnn$ such that $\widehat{A}=\operatorname*{Op}\nolimits_{\mathrm{BJ}}(a)$.
\end{theorem}

Concerning the injectivity of the map \eqref{map}, we begin with a simple example, which shows that the map \eqref{map} is not ono to one, even when restricted to real analytic functions which extend to entire functions in $\mathbb{C}^{2n}$.
\begin{example}\label{rem2}\rm 
Consider the Born-Jordan symbol 
\begin{equation}\label{eq2}
a(z)=e^{\frac{i}{\hslash}\sigma(z_0,z)}=e^{\frac{i}{\hslash}(p_0 x-x_0 p)},
\end{equation} where $z_0=(x_0,p_0)$ is any point on the zero set $\Theta(z)=0$.\par
 The symplectic Fourier transform of $a$ is
\[
a_\sigma(z)=(2\pi\hslash)^n \delta(z-z_0)
\]
and therefore $\Theta a_\sigma=0$, because $\Theta(z_0)=0$. Hence the corresponding Weyl symbol is $a_{\rm W}=0$ by \eqref{eq11}.\par
Observe that the symbol $a$ in \eqref{eq2} extends to an entire function $a(\zeta_1,\zeta_2)=e^{\frac{i}{\hslash}(p_0 \zeta_1-x_0 \zeta_2)}$ in $\mathbb{C}^{2n}$, satisfying the estimate
\[
|a(\zeta)|\leq \exp\Big(\frac{r}{\hslash}|{\rm Im}\,\zeta|\Big),\quad \zeta\in \mathbb{C}^{2n},
\] 
where $r=|(p_0,-x_0)|=|z_0|$.\par
For future reference we observe that the minimum value of $r$ is reached for the points $z_0$ at which the hypersurface $\Theta(z)$ has minimum distance from $0$, which turns out to be $r=\sqrt{4\pi\hslash}$. For example one can consider $x_0=p_0= (2\pi \hslash/n)^{1/2}(1,\ldots,1)$, so that $x_0 p_0=2\pi\hslash$ and $|x_0|^2+|p_0|^2=4\pi\hslash$ (see Figure 4).
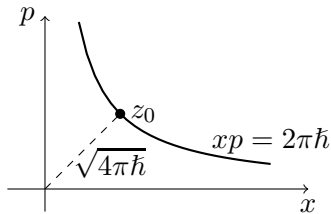
\begin{figure}[b]
\begin{tikzpicture}
\draw[->] (-0.5,0) -- (3.5,0) node[below]{$x$};
\draw[->] (0,-0.3) -- (0,2.3) node[left]{$p$};
\draw[thick, domain=0.45:3] plot (\x, {1/\x}) node[above] {$xp=2\pi\hslash$};
\draw[dashed] (0,0) -- (1,1);
\node at (0.86,0.35) {$\sqrt{4\pi\hslash}$};
\filldraw (1, 1) circle (1.7pt) node[right]{$z_0$};
\end{tikzpicture}
\caption{The point $z_0$ minimizes the distance of the zero set $\Theta(z)=0$ from $0$.}
\end{figure}

\end{example}
Inspired by the above example we now exhibit a non-trivial class of functions on which the map \eqref{map} in injective.
\begin{definition}\label{defar}
For $r\geq0$, let $\mathcal{A}_r$ be the space of smooth functions $a$ in $\rnn$ that extend to entire functions $a(\zeta)$ in $\mathbb{C}^{2n}$ and satisfying the estimate
\[
|a(\zeta)|\leq C(1+|\zeta|)^N\exp\Big(\frac{r}{\hslash}|{\rm Im}\,\zeta|\Big),\quad \zeta\in \mathbb{C}^{2n},
\]
for some $C,N>0$. \par
Equivalently (by the Paley-Wiener-Schwartz Theorem) $\mathcal{A}_r$ is the space of temperate distributions in $\rnn$ whose (symplectic) Fourier transform is supported in the closed ball $|z|\leq r$. 
\end{definition}
\begin{remark}
Observe that the space $\mathcal{A}_0$ is just the space of polynomials in phase space.
\end{remark}  
We have the following result. 
\begin{proposition}\label{pro6}
The map \eqref{map} is a bijection $\mathcal{A}_r\to\mathcal{A}_r$ if and only if $
0\leq r<\sqrt{4\pi\hslash}.
$
\end{proposition}
\begin{proof}
The ``only if'' part follows at once from the Remark \ref{rem2} because the symbol $a(z)$ in \eqref{eq2} belongs to $\mathcal{A}_r$ for $r\geq\sqrt{4\pi\hslash}$ and is mapped to $0$. \par
Consider now the ``if'' part. Taking the symplectic Fourier transform in \eqref{map} and by the Paley-Wiener-Schwartz Theorem we are reduced to prove that, when $0\leq r<\sqrt{4\pi\hslash}$, the map
\[
a\longmapsto \Theta a 
\]
is a bijection $\mathcal{E}'(B_{r})\to  \mathcal{E}'(B_{r})$, where $\mathcal{E}'(B_{r})$ is the space of distributions on $\rnn$ supported in the closed ball $B_{r}$ given by $|z|\leq r$.\par
 Now, it is clear that if $a\in \mathcal{E}'(B_{r})$ then $\Theta a\in \mathcal{E}'(B_{r})$. On the other hand, since the function $\Theta(z)$ does not vanish for $|z|< \sqrt{4\pi\hslash}$,  hence in a neighborhood of $B_{r}$ (by assumption $r<\sqrt{4\pi\hslash}$), the equation $\Theta a=b$, for every $b\in \mathcal{E}'(B_{r})$, has a unique solution $a\in \mathcal{E}'(B_{r})$ obtained simply by multiplying by $\Theta^{-1}$: $a=\Theta^{-1} b$. 
\end{proof}
\begin{remark}
The above result recaptures and generalizes the fact that the map \eqref{map} is a bijection of the space of polynomials in phase space into itself (case $r=0$); see Section \ref{polynomial}. 
\end{remark}

We now study the surjectivity of the map \eqref{map} on the spaces $\mathcal{A}_r$ when $r\geq \sqrt{4\pi\hslash}$.
 
\begin{theorem}\label{mainteo2}
Let $ r\geq0$. For every $b\in \mathcal {A}_r$ there exists $a\in \mathcal {A}_{r}$ such that
\[
\left(\tfrac{1}{\tph}\right)^n a \ast \Theta_\sigma=b.
\]
\end{theorem}
\begin{proof}
As above we have to prove that the equation 
\[
 \Theta a=b
\]
admits at least a solution $a\in \mathcal{E}'(B_r)$, for every $b\in \mathcal{E}'(B_r)$.\par
Since all the distributions here are compactly supported, the problem is local and we can solve the equation $\Theta a=b$ in $\mathcal{E}'(U_{z_0})$ for a sufficiently small open neighborhood $U_{z_0}$ of any given point $z_0$ and conclude with a finite smooth partition of unity.\par
If $|z_0|>r$ and $U_{z_0}\subset \{|z|>r\}$ one can choose $a=0$ in $U_{z_0}$.\par\par
When $|z_0|<r$ and $U_{z_0}\subset \{|z|<r\}$ we apply the classical division theorem valid for smooth functions with at most simple zeros \cite[page 127]{schwartz}: for every $b\in \mathcal{E}'(U_{z_0})$ there therefore exists a solution $a\in \mathcal{E}'(U_{z_0})$.\par
Of course if $\Theta(z_0)\not=0$ the division is trivial, so that we now suppose that $z_0=(x_0,p_0)$ belongs to both $|z|=r$ and $xp= 2\pi k\hslash$ for some $k\in\mathbb{Z}$, $k\not=0$. Then necessarily we have $r\geq\sqrt{4\pi|k| \hslash}$, because this is the distance of the hypersurface $xp= 2\pi k\hslash$ from the origin. We therefore distinguish two cases. \par
{\it First case}:  $r>\sqrt{4\pi |k|\hslash}$. Then the hypersurfaces  $|z|=r$ and $xp= 2\pi k\hslash$ cut transversally at $z_0$, i.e.\ their normal vectors are linearly independent and the intersection $\Sigma$ is therefore a submanifold of codimension 2. In fact one sees easily that the vector normals to these two hypersurfaces at $z_0$ are linearly dependent if and only if $p_0={\rm sign}(k) x_0$ and $|x_0|^2=2\pi |k|\hslash$. In that case we must have $r=\sqrt{4\pi |k|\hslash}$.\par
{\it Second case}: $r=\sqrt{4\pi |k|\hslash}$. Then the hypersurfaces $|z|=r$ and $xp=2\pi k\hslash$ touch along the submanifold $\Sigma$ of codimension $n+1$ having equations $p={\rm sign}(k) x$, $|x|^2=2\pi |k|\hslash$. \par
In both cases by the implicit function theorem we can take analytic coordinates $y=(y',y_{2n})$ near $z_0$ so that $z_0$ has coordinates $y=0$, the hypersurface $xp= 2\pi k\hslash$ is straightened to $y_{2n}=0$ and moreover the above submanifold $\Sigma$ has equations $y_1=y_{2n}=0$ (in the first case) or $y_1=\ldots =y_{n}=y_{2n}=0$ (in the second case). The portion of ball $|z|\leq r$ near $z_0$ is defined now by the inequality $y_{2n}\geq f(y')$ for some real-analytic function $f(y')$ defined in a neighborhood of $0$, and vanishing on $\Sigma\ni0$.\par
 Hence we are reduced to solve the equation 
\[
y_{2n} a=b
\]
in a neighborhood of $0$, where $b$ is supported in the set  $y_{2n}\geq f(y')$ and we look for $a$ supported in the same set. This is exactly the situation of Proposition \ref{prolo} (possibly after a rescaling). As already observed, the condition \eqref{lo} is satisfied by every real-analytic function and therefore Proposition \ref{prolo} gives the desired conclusion.
\end{proof}
\begin{example}\label{esempio}\rm
We want to find a Born-Jordan symbol of the operator $\widehat{T}(z_0)$ in \eqref{HW}, $z_0=(x_0,p_0)\in\rnn$. First of all we observe that $\widehat{T}(z_0)$ has Weyl symbol
\[
b(z)=e^{\frac{i}{\hslash}\sigma(z_0,z)};
\]
see \cite[Proposition 198]{Birkbis}. 
Hence we are looking for $a\in\spnn$ such that 
$
\left(\tfrac{1}{\tph}\right)^n a\ast\Theta_\sigma=b,
$
or equivalently, taking the symplectic Fourier trasform, $\Theta a_\sigma=b_\sigma$, that is
\begin{equation}\label{e5}
\Theta(z) a_\sigma(z)=(2\pi \hslash)^n\delta(z-z_0).
\end{equation}
Now, if $\Theta(z_0)\not=0$ we can take $a_\sigma(z)= \Theta(z_0)^{-1}(2\pi \hslash)^n\delta(z-z_0)$, namely $a(z)=\Theta(z_0)^{-1}e^{\frac{i}{\hslash}\sigma(z_0,z)}$. \par
If instead $\Theta(z_0)=0$ we look for $a_\sigma$ in the form 
\begin{equation}\label{e6}
a_\sigma(z)=(2\pi \hslash)^n\sum_{j=1}^{2n}c_j\partial_j\delta(z-z_0)
\end{equation}
for unknown $c_j\in\mathbb{C}$, $j=1,\ldots, 2n$. Since $\Theta(z_0)=0$ we have
\[
\Theta(z) a_\sigma(z)= (2\pi \hslash)^n\Big(-\sum_{j=1}^{2n} c_j \partial_{j}\Theta(z_0)\Big)\delta(z-z_0),
\]
so that the equation \eqref{e5} reduces to
\[
-\sum_{j=1}^{2n} c_j \partial_{j}\Theta(z_0)=1
\]
which has infinitely many solutions, because $\nabla \Theta(z_0)\not=0$ if $\Theta(z_0)=0$ ($\Theta(z)$ has only simple zeros). For any solution $\mathbf{c}:=(c_1,\ldots, c_{2n})$, taking the inverse symplectic Fourier transform in \eqref{e6} (using the formulas $(a_\sigma)_\sigma=a$, $a_\sigma(z)=F a(Jz)$, $F(\partial_j a)=\frac{i}{\hslash}z_j Fa$), we find a Born-Jordan symbol
\[
a(z)= \frac{i}{\hslash} \sigma(z,\mathbf{c})e^{\frac{i}{\hslash}\sigma(z_0,z)}.
\]
Observe that $b\in \mathcal{A}_r$ with $r=|z_0|$ and $a\in \mathcal{A}_r$ as well.

\end{example}

\end{document}